\documentclass[11pt, reqno]{article}

\usepackage{fullpage}
\usepackage{enumerate}
\usepackage{setspace}

\usepackage{upgreek, mathtools,bm}

\usepackage{amsmath,amsfonts,amssymb,graphics,amsthm, comment}
\usepackage{hyperref}
\hypersetup{
    colorlinks=true,
    linkcolor=blue,
    citecolor=red,
    urlcolor=blue,
    pdfborder={0 0 0}
}
\usepackage{yfonts}

\usepackage{graphicx}
\usepackage{xcolor}
\usepackage{bbm}
\usepackage{wrapfig} 

\usepackage[font=sf, labelfont={sf,bf}, margin=1cm]{caption}

\usepackage{cleveref}
  \crefname{theorem}{Theorem}{Theorems}
  \crefname{thm}{Theorem}{Theorems}
  \crefname{thm*}{Theorem*}{Theorems}
  \crefname{lemma}{Lemma}{Lemmas}
  \crefname{lem}{Lemma}{Lemmas}
  \crefname{remark}{Remark}{Remarks}
  \crefname{prop}{Proposition}{Propositions}
\crefname{notation}{Notation}{Notations}
\crefname{claim}{Claim}{Claims}
  \crefname{defn}{Definition}{Definitions}
  \crefname{corollary}{Corollary}{Corollaries}
  \crefname{section}{Section}{Sections}
  \crefname{figure}{Figure}{Figures}
    \crefname{assumption}{Assumption}{Assumptions}

\newtheorem{thm}{Theorem}[section]
\newtheorem{thm*}{Theorem*}[section]

\newtheorem{lemma}[thm]{Lemma}
\newtheorem{corollary}[thm]{Corollary}
\newtheorem{prop}[thm]{Proposition}
\newtheorem{defn}[thm]{Definition}

\numberwithin{equation}{section}

\theoremstyle{definition}
\newtheorem{remark}[thm]{Remark}

\def\cT{\mathcal{T}}

\def\cR{\mathcal{R}}
\def\cQ{\mathcal{Q}}

\def\cL{\mathcal{L}}

\def\cH{\mathcal{H}}
\def\cG{\mathcal{G}}

\def\cD{\mathcal{D}}
\def\cC{\mathcal{C}}
\def\cB{\mathcal{B}}

\def \ve {\varepsilon}

\def\P{\mathbb{P}}

\def\E{\mathbb{E}}
\def\C{\mathbb{C}}
\def\R{\mathbb{R}}
\def\Z{\mathbb{Z}}

\def  \p- {p\textunderscore}

\def \d {{\# \delta}}

\def\Euc{\textsf{euc}}

\DeclareMathOperator{\diam}{Diam}

\newcommand{\abs}[1]{ \lvert #1 \rvert}

\newcommand{\crsw}{c_{\mathsf{RSW}}}

\def \Ce {C_{\textsf{Euc}}}

\begin{document}
 \title{Quantitative Russo--Seymour--Welsh for random walk on random graphs and decorrelation of UST}

\author{Gourab Ray\thanks{University of Victoria, supported in part by NSERC 50311-57400, \textsf{gourabray@uvic.ca}} \and Tingzhou Yu  \thanks{University of Victoria, \textsf{tingzhou.yui@gmail.com}}}

\maketitle
\abstract{We prove a quantitative Russo--Seymour--Welsh (RSW) type result for random walks on two natural examples of random planar graphs: the supercritical percolation cluster in $\Z^2$ and the Poisson Voronoi triangulation in $\R^2$. More precisely, we prove that the probability that a simple random walk crosses a rectangle in the hard direction with uniformly positive probability is stretched exponentially likely in the size of the rectangle. As an application we prove a near optimal decorrelation result for uniform spanning trees for such graphs. This is the key missing step in the application of the proof stretegy of \cite{BLR16} for such graphs (in \cite{BLR16}, random walk RSW was assumed to hold with probability 1). Applications to almost sure Gaussian free field scaling limit for dimers on Temperleyan type modification on such graphs are also discussed.}
\section{Introduction}\label{sec:intro}
It is now an established fact that the Russo--Seymour--Welsh theory lies at the heart of two dimensional statistical physics models, particularly those which are believed to be conformally invariant in their scaling limits. Although the major application of this idea has been in percolation theory (see e.g. \cite{grimmett1999percolation,duminil2017lectures,duminil_sixty} for a broad overview), recently, this concept was used in \cite{BLR16} to study decorrelation of uniform spanning trees in a  general setting. Very roughly, this type of estimate leads to rough Harnack type inequality and also Beurling type hitting estimates. This ultimately led to a result which establishes scaling limits of dimer height functions to a Gaussian free field on a fairly general class of graphs. This is later extended to graphs on multiply connected Riemann surfaces as well \cite{BLR_riemann1,BLR_riemann2}.

In this program, two key assumptions are made on the graph. The first assumption is that random walk on the graph must converge to Brownian motion, which can be thought of as an assumption on the macroscopic symmetries of the random walk under conformal mappings.  This assumption is usually robust under reasonable perturbations of the underlying graph. The second assumption is that for any rectangle larger than a \emph{fixed scale}, a random walker crosses it without exiting the rectangle with a probability uniform in the scale and location, and depending only on the aspect ratio (see \cref{def:RSW}). The second assumption was called a Russo--Seymour--Welsh (RSW) type assumption in \cite{BLR16}. It can be checked that this holds for all standard lattices (more generally for isoradial graphs with unifomly elliptic angles), and even holds if we put some uniformly elliptic random environment on them (see \cite[Section 1.1]{BLR16} for a detailed discussion). Let us remark that the RSW assumption is in some sense related to the uniformity in the rate of convergence of the random walk to a Brownian motion depending on the location of the graph. Indeed, for this reason, RSW for the square lattice for example is a simple consequence of the invariance principle. On the other hand in the presence of some local irregularities, it is not  clear at all if such an estimate is even true.

 The goal of this article is to extend the random walk RSW result to random planar graphs (with a natural embedding in $\Z^2$) which are not necessarily `uniformly elliptic' in the sense of the examples considered so far. The key examples we handle in this article are the unique infinite cluster of a Bernoulli bond percolation on $\Z^2$ and a Poisson Voronoi triangulation. It can be easily seen that in both these cases, RSW does not hold deterministically for rectangles larger than any fixed scale uniformly over the location of the graph (for example, an arbitrarily large rectangle is empty at some location almost surely). However, we show in \cref{thm:RSW_Poisson,thm:RSW_perc} that RSW holds with a probability which is stretched exponentially high in the scale. 
 
 One major application of the RSW assumption in \cite[Theorem 4.21]{BLR16} was to prove a macroscopic decorrelation result for uniform spanning trees in the following sense. Suppose $D $ is a simply connected domain and we take a Uniform spanning tree $\cT$ with wired boundary condition for a graph with small `mesh size' $\delta$ in it. Fix two points $x,y$ in $D$. Then one can couple two independent USTs $\cT_1, \cT_2$ (on a possibly bigger domain) with $\cT$ so that $\cT_1$ and $\cT$ agree on a small but random neighbourhood of $x$, and $\cT_2$ and $\cT$ agree on a small but random neighbourhood of $y$. Furthermore, this random neighbourhood has macroscopic radius, in the sense that the radius dominates a random variable which is independent of $\delta$ for all small enough $\delta$. Furthermore, one can obtain a polynomial bound on the lower tail of the radius. The same result holds not just for two but for any finite number of points. In this article, we extend this result to UST on random graphs in \cref{thm:coupling}. In particular, we show that a similar coupling can be obtained for a collection of graphs which has high probability (this can be extended to an a.s. statement along a subsequence of $\delta$, see \cref{a.s.})
 
 Another consequence of the quantitative RSW and \cref{thm:coupling} is a scaling limit result for dimer height function on such random planar graphs. Suppose we take a random planar graph satisfying the quantitative RSW and some other mild assumptions (which are satisfied by the infinite cluster of Bernoulli bond percolation and $\Z^2$ and the Poisson Voronoi triangulation). There is a natural way to add a dual to this graph so that we obtain a \emph{Temperleyan graph} which admits a dimer cover. Then following the stratregy of \cite{BLR16}, it can be shown that almost surely on the graph the height function of this dimer model converges to Gaussian free field. More details and discussions on this can be found in \cref{sec:dimer}.
 
 The main input for the random walk RSW results is a result by Barlow \cite{Barlow_perc}, which states that a quadratic volume growth and Poincar\'{e} inequality ensures a good heat kernel bound for random walks. The Poincar\'{e} inequality is an analytic criterion. One can establish this inequality with a good control on the volume growth and isoperimetric constant on the graph. We collect these geometric criterions in \cref{lem:crossing_criterion}. We need to be a bit more careful than the treatment in Barlow \cite{Barlow_perc}  as there is no uniform bound on degree in our assumption (which was assumed in \cite{Barlow_perc} as the main motivation there was to study heat kernel bounds for percolation in $\Z^d$). Next, we show that the criterions in \cref{lem:crossing_criterion} hold for the main two applications in this article: unique infinite cluster of Bernoulli percolation and Voronoi triangulation. A key input is a quantitative isoperimetric inequality for Voronoi triangulation (\cref{lem:iso_2}), which we consider to be another novel contribution of this article. The coupling result of \cref{sec:UST} is established by first proving that the diameters of the `bad regions' where RSW does not hold is small with high probability, and then applying the techniques established in \cite{BLR16}. Finally in \cref{sec:dimer} we explain how the results in this article can be used to prove a Gaussian free field scaling limit of the dimer model on random graphs, which hold a.s. on the graph chosen.

\paragraph{Acknowledgement:} We thank Benoit Laslier for several useful discussions.

\paragraph{Notation:} Given a graph $G$, we denote by $V(G)$ its vertex set and $E(G)$ its edge set. Let $\Lambda_n$ denote the square $[-n,n]^2$ with $\Lambda_n(x) = x+\Lambda_n$. Sometimes we will also deal with rectangles $\Lambda_{m,n} = [-m,m] \times [-n,n]$ and similarly $\Lambda_{m,n}(z) = z+\Lambda_{m,n}$. For $S \subset V(G)$, denote by $|S|_E$ the sum of the weights of the edges incident to vertices of $G$ in $S$, while $|S|$ simply denotes the number of vertices in $S$. With a slight abuse of notation, for any $S \subset \R^2$, we use $|S|_E$ to denote the sum of the weights of the edges which  is incident to a vertex in $S$, and $|S|$ to denote the cardinality of the set of vertices in $S$. Let $\deg (v)$ denotes the degree of the vertex $v$.

\section{A general criterion for RSW}\label{sec:RSW_general}
In this section, we work with a fixed graph $G =(V,E)$ embedded in $\R^2$. The goal of this article is to summarize certain geometric properties of the graph which ensures that a simple random walk on it behaves in a nice manner. The main quantity of interest is the following.

\begin{defn}\label{def:RSW}
For $c>0$, we say that $\Lambda_{3m,m}(z)$ is \textbf{$c$-crossable} if for every $x \in B_1:=\Lambda_{m/2}(z+(-2m,0))$ and  with $ B_2:=\Lambda_{m/2}(z+ (2m,0))$, 
\begin{equation}
\P(Y \text{ started from $x$ enters $B_2$ before exiting $\Lambda_{3m,m}(z)$}) \ge c.\label{eq:C-crossable}
\end{equation}
where $Y$ is a simple random walk.
\end{defn}
See \cref{fig:crossing}. Notice that although the event is stated in terms of a random walk, it is in fact a statement about the geometry of the graph inside the rectangle. We point out that in \cite{BLR16}, it was assumed that there exists a $c>0$ depending only on the graph such that every rectangle beyond a certain scale $m \ge \delta_0^{-1}$ centred at any $z$ was $c$-crossable, and this assumption was called `\emph{uniform crossing}'. 
\begin{figure}[h]
\centering
\includegraphics[width = 0.5\textwidth]{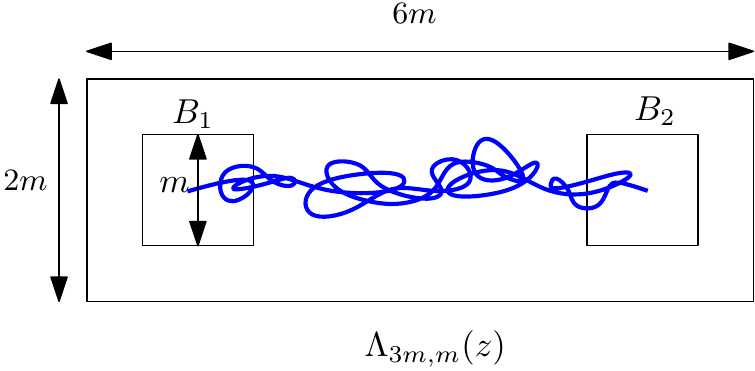}
\caption{The rectangle is $c$ crossable if an event like the above has probability at least $c$.}\label{fig:crossing}
\end{figure}

Let $G = (V, E)$ be a finite planar graph, properly embedded in $\R^2$. For a function $f:V\mapsto \R$, denote by $\nabla f$ to be a function from the oriented edges of the graph to $\R$, satisfying $$\nabla f((e_-, e_+)) = f(e_+ ) - f(e_-).$$
We will denote by $|\nabla f|$ the function which takes absolute value of $\nabla f$ for each unoriented edge in $E$. We now borrow the notions of `good' and `very good' from \cite{Barlow_perc}. Let $o$ denote the vertex in $G$ which is closest to the origin in $\R^2$. Also let $d_G$ denote the graph distance in $G$ and let $B(x,r)$ denote the graph distance ball of radius $r$ centred around the vertex $x$.

\begin{defn} (\cite[Definition 1.7]{Barlow_perc})\label{def:good}
Let $C_P, C_V>0$ and $C_W \ge 1$ be fixed.
We say that $B(o,n)$ is $(C_P, C_V, C_W)$-\textbf{good} if it satisfies
\begin{equation}\tag{Vol}
|B(o, n)|_E \ge C_Vn^2.\label{eq:vol}
\end{equation}
and every $f :B(o,C_W n) \mapsto \R$ satisfies the \textbf{weak Poincar\'e inequality}, i.e.,
\begin{equation}\tag{P}
\sum_{v \in B(o,n) } (f(v) - \bar f)^2 \deg (v) \le C_Pn^2 \sum_{e \in E(B(o,C_Wn))} \abs{\nabla f (e)}^2 \label{eq:WP}
\end{equation}
where $\bar f = |V(B(o,n)) |_E^{-1}\sum_{v \in B(o,n)} f(v)\deg(v) $.
We say $B(o,n)$ is $(C_P, C_V, C_W)$-\textbf{very good} if there exists an integer $N_{B(o,n)} \le n^{1/4}$ such that every $B(y,r) \subset B(o,n)$ is good for every $N_B \le r\le n$. 
\end{defn}
We will sometimes drop the constants in the definition of good and very good when they are clear from the context.

%
%

We now extend the notion of good and very good to Euclidean squares. We say $\Lambda_n$  is $(\Ce, C_P,C_V,C_W)$-very good if and only if \begin{equation}
B(o, \Ce^{-1} n) \subset \Lambda_n \subset B(o, \Ce n)\label{eq:inclusion_good}
\end{equation}
and $B(o, \Ce n)$ is $(C_P,C_V,C_W)$-very good. The above inclusion is interpreted as follows: if a vertex of $G$ is outside  $B(o, \Ce n)$ (resp. inside $B(o, \Ce^{-1} n)$) then it is outside (resp. inside) $\Lambda_n$.
\begin{lemma}\label{lem:RSW_condition}
Suppose there exist constants $\Ce,C_P,C_V,C_W,c_0, d$ such that for all $n \ge 1$ the following is true.
\begin{enumerate}[a.]
\item $\Lambda_{n \log n}$ is $(\Ce,C_P,C_V,C_W)$-very good with $N_{\Lambda_{n \log n}} \le n^{1/8}$.
\item $|B_2| \ge dn^2$ where $B_2$ is as in \cref{def:RSW} and $m = n/4$.
\item The graph distance between any vertex in $\Lambda_{2.5m, m/2}$ and any vertex outside $  \Lambda_{3m,m}$ is at least $c_0 n$ and $m = n/4$.
\end{enumerate}
 Then there exists a constant $c = c(C_P, C_V, C_W, \Ce, c_0,d)$ such that $\Lambda_{3m,m}$ is $c$-crossable.
\end{lemma}
\begin{proof}
Let $Y$ be the continuous time random walk with $q_t^0$ denoting its density killed upon exiting $\Lambda_{3m,m}$. More precisely, denoting $\tau$ to be the infimum over times when $Y$ is not in $\Lambda_{3m,m}$, 
$$
q_t^0(x,y) = \P(Y_t = y, Y_0 = x, \tau>t); \qquad x,y \in \Lambda_{3m,m}.
$$ 
Using the same  argument as in \cite[Lemma 5.8]{Barlow_perc}, we can show that there exists a constant $c = c(C_P, C_V, C_W, \Ce, c_0)$ such that for any $x\in B_1$ and $y \in B_2$,
$$
q_t^0(x,y) \ge \frac{c}{t} \text{ for all }cn^2 \le t \le c^{-1}n^2
$$
Let us provide some details of this fact. We can write for any $z \in \Lambda_{2.5m,m/2}$ and any $t>0$,
\begin{align*}
q_t^0(x,z) & \ge 	q_t(x,z)  -\E_x (1_{\tau<t} q_{t-\tau}(Y_{\tau, z}))\\
& \ge q_t(x,z) - \sup_{0 \le s \le t} \sup_{w \in \partial'} q_s(w,z).
\end{align*}
where $\partial '$ denote the set of vertices outside $\Lambda_{3m,m}$ with at least one neighbour in $\Lambda_{3m,m}$. Now fix $\ve \in (0,1/8)$, $z \in B(x,\delta n)$ and $t = \delta^2 n^2$. This allows us to apply the heat kernel bound of \cite[Theorem 5.3]{Barlow_perc} to lower bound the first term $q_t(x,z)$ by   $t^{-1}ce^{-c'}$ (for ease of reference, we point out that we choose $x_0=x_1 = o$, $R\log R = \Ce n\log n$ in the notations of that theorem). On the other hand since the graph distance between any $w \in \partial '$ and $z \in \Lambda_{2.5m,m/2}$ is at least $c_0 n$ by the third item above, we can use \cite[Theorem 3.8]{Barlow_perc} to upper bound the second term. Namely, writing $s = \theta t$,
\begin{equation*}
 \sup_{0 \le s \le t} \sup_{w \in \partial'} q_s(w,z) \le \sup_{ 0 \le \theta \le 1}\frac1{\theta t} e^{-\frac{c'c_0^2n^2}{\theta \delta^2 n^2}} = \sup_{ 0 \le \theta \le 1}\frac1{\theta} e^{-\frac{c'c^2_0}{\theta \delta^2}} .
\end{equation*}
(Again for ease of reference, we point out that we choose $x_0 = o$, $R  = \Ce n\log n$ in the notations of  \cite[Theorem 3.8]{Barlow_perc}.)
Actually to be more precise for very small values of $\theta $ we go outside the range of times when \cite[Theorem 3.8]{Barlow_perc} is applicable, in which case we use the upper bound \cite[Lemma 1.1]{Barlow_perc} instead. Choosing $\delta$ small enough we obtain 
\begin{equation*}
q_t^0(x,z) \ge \frac{c(\delta)}{t} = \frac{c(\delta)}{\delta^2 n^2}.
\end{equation*}
Now we use the standard chaining argument. Namely, we choose a sequence of balls inside $\Lambda_{2.5m,m}$ of Volume $O(\delta^2 n^2)$ such that on the event that the walk iteratively lands on these sequence of balls without leaving $\Lambda_{2.5m,m}$, the walk enters $B_2$. Using the Markov property of the walk,  this event has probability at least $c(\delta)^{O(\delta^{-1})}$. 
The proof is complete as the probability of crossing is at least this constant.
\end{proof}
%

In light of \cref{lem:RSW_condition}, it is clear that in the setting of random graphs, we require a box to be very good with high probability. To that end, it is useful to find a geometric condition for a box being very good. While \eqref{eq:vol} is a very simple geometric condition, \eqref{eq:WP} is analytic. We present below a lemma which essentially states that a relevant Isoperimetric inequality implies \eqref{eq:WP}.

We now recall some relevant definitions regarding isoperimetry of general graphs. Take a finite, connected graph $H$. For any $A \subset V(H)$, let 
$$
i_H(A) := \frac{\partial_E(A, H \setminus A)}{|A|_E}.
$$
where $\partial_E(A, H \setminus A)$ denotes the collection of edges with one endpoint in $A$ and another in $H \setminus A$. We say $A$ is connected if the subgraph induced by $A$ is connected. Define the isoperimetric constant $I_H$ as 
$$
I_H := \inf \{i_H(A): 0<|A|_E \le \frac12 |H|_E : A \text{ and }H \setminus A \text{ are connected}\}.
$$
A subgraph $H'$ of $H$ is a graph induced by a subset of vertices of $H$. Notice that the definition of $I_{H'}$ ignores the edges not present in $E(H')$. Observe the assertion of connectedness in $A$ and $H \setminus A$ being connected is slightly non-standard, however \cite[Lemma 1.3]{Barlow_perc} ensures that removing this connectedness assertion only changes the constant by a factor of 2. This leads us to the following rephrasing of \cite[Proposition 1.4(a)]{Barlow_perc}:
\begin{lemma}\label{lem:iso_poincare}
Then there exists $c>0$ such that for any  subgraph $H $ of $G$ and any function $f:V(H) \mapsto \R$, 
\begin{equation*}
\sum_{v \in V(H) } (f(v) - \bar {f_H})^2 \deg_H (v) \le \frac{c}{I^2_H} \sum_{e \in E(H)} \abs{\nabla f (e)}^2 \label{eq:iso_p}
\end{equation*}
where $\bar {f_H} = |V(H) |_{E(H)}^{-1}\sum_{v \in V(H)} f(v) \deg_H(v) $ and $\deg_H(v)$  is the degree of $v$ and $|V(H) |_{E(H)}$ is the sum over degrees of vertices in $V(H)$ counting edges only in $H$.
\end{lemma}
This allows us to describe an equivalent geometric criterion which will ensure a crossing estimate. 

\begin{lemma}\label{lem:crossing_criterion}
Suppose there exist constants $\Ce, C_V, C_V' ,C_W, C_I, d ,c_0$ such that for all $ n \ge 1$,
\begin{enumerate}[{(}i{)}]
\item $B(o, \Ce^{-1} n\log n) \subset \Lambda_{n\log n} \subset B(o, \Ce n\log n)$.\label{eq:VolEuc}
\item For all $n^{1/9} \le r \le \Ce n \log n+1$, $$C_V r^2 \le |B(y,r)|_E \le C_V' r^2$$ for all $B(y, r) \subset B(o, \Ce n\log n+1)$.
\item For all set of vertices  $A \subseteq V(B(y, r)) \subseteq V(B(o, \Ce n \log n+1))$ inducing a connected subgraph such that $n^{1/9} \le |A|_E \le  \frac{ |B(y,r)|_E}{2},$
\begin{equation}
i_{B(y,r)}(A) \ge \frac{C_I}{\sqrt{|A|_E}} \label{eq:i_A}
\end{equation}
\item $|B_2|_E \ge dn^2$ where $B_2$ is as in \cref{eq:C-crossable}.
\item Let $m = \frac{3n}{12}$. The graph distance between any vertex in $\Lambda_{2.5m, m/2}$ and any vertex outside $ \Lambda_{3m,m}$ is at least $c_0 n$.
\end{enumerate}
Then there exists a constant $c>0$ (depending only on the constants above) such that for all $n \ge 1$, $\Lambda_{3m,m}$ is $c$-crossable.
\end{lemma}
\begin{proof}
This is an application of \cref{lem:RSW_condition,lem:iso_poincare}. Indeed, if $B(o, \Ce n\log n)$ is very good then items (i), (iv) and (v) imply c-crossability by \cref{lem:RSW_condition}. The lower bound of item (ii) establishes \eqref{eq:vol}, so we only need to establish \eqref{eq:WP} for every $B(y,r) \subset B(o,\Ce n \log n)$ for $n^{1/9} \le r \le \Ce n \log n$ (i.e., we choose $N_{B(o, \Ce n \log n)} = n^{1/9}$). 

Fix $y$ with $B(y,r) \subseteq B(o,\Ce n \log n)$, and write $B_r$ for $B(y,r)$ to minimize notation.
Notice that for any $B_r \subseteq B(o,\Ce n \log n)$ with $n^{1/9} \le r \le \Ce n \log n$ and any connected set $A \subseteq B_{r+1}$ with $n^{1/9}<|A|_E \le |B_{r+1}|_E/2$, we have 
$$i_{B_{r+1}}(A )\ge C_I|A|^{-1/2}_E \ge \frac{C_I\sqrt{2}}{\sqrt{C_V'} (r+1)} $$ by the upper bound of item (ii) and the isoperimetric inequality (iii). On the other hand, if $|A|_E \le n^{1/9}$, then trivially $i_{B_{r+1}}(A) \ge |A|^{-1}_E \ge n^{-1/9} \ge r^{-1}$. Thus, by possibly decreasing $C_I$ and increasing $C'_V$ if needed, we obtain that $$I_{B_{r+1}} \ge \tilde C r^{-1}$$ with $\tilde C =C_I\sqrt{2/C'_V}$. 

 Now choose $C_W = 2$, and for any function $f:B_{2r} \mapsto \R$. Restrict $f$ to the graph induced by $B_r$ union all the vertices in $\partial_E (B_r, G \setminus B_r)$. 
\begin{align*}
\sum_{v \in V(B_r) } (f(v) - \bar {f})^2 \deg_{G} (v) & \le \sum_{v \in V(B_r) } (f(v) - \bar {f}_{B_{r+1}})^2 \deg_{G} (v) \\
& \le \sum_{v \in V(B_{r+1}) } (f(v) - \bar {f}_{B_{r+1}})^2 \deg_{B_{r+1}} (v)\\
& \le \frac{cr^2}{\tilde C^2}\sum_{e \in E(B_{r+1})} \abs{\nabla f (e)}^2.
\end{align*}
where $c$ is as in \cref{lem:iso_poincare}. Notice that cannot directly use \cref{lem:iso_poincare} as the degrees are counted in $G$ which could potentially be large as we have no assumption on the degree bound. The first inequality follows from the fact that $\bar {f}$ is the minimum over $a$ of $\sum_{v \in V(B_r) } (f(v) - a)^2 \deg_{G} (v)$. The second inequality is a trivial addition of nonnegative terms along with the fact that $\deg_G(v) = \deg_{B_{r+1}(v)}$ if $v \in B_r$. The final inequality follows from \cref{lem:iso_poincare} applied to $B_{r+1}$. Thus, we have established \eqref{eq:WP} with $C_P = c\tilde C^{-2}$ since adding $|\nabla f(e)|^2$ over the rest of the edges of $B_{2r}$ only increases the right hand side.
\end{proof}

\section{RSW for Bernoulli Percolation}\label{sec:bern}

In this section, we focus on \text{Bernoulli bond percolation} on $\Z^2$.  For $p \in [0,1]$ let $\P_p$ denote the Bernoulli bond percolation probability measure induced by i.i.d.\ coin flips, one for each edge of $\Z^2$. We call an edge \textbf{open} if the edge is present, and \textbf{closed} otherwise. A cluster denotes a connected component of open edges.
It is well known that for $p>p_c:=1/2$ (i.e. the percolation is supercritical) there exists a unique infinite cluster almost surely (see e.g. \cite{grimmett1999percolation}), call it $\cC_\infty$. We refer to \cite{duminil_sixty} for relevant history and references of this very popular model. Our main result in this section is an RSW type result for random walk on $\cC_\infty$. Recall the definition of $c$-crossable from \eqref{eq:C-crossable}.

\begin{thm}\label{thm:RSW_perc}
Fix $p>p_c=1/2$ and let $\cC_\infty$ be the unique infinite cluster for supercritical bond percolation in $\Z^2$ induced by the probability measure $\P_p$. There exist constants $c_p, c,\alpha \in (0,1]$ such that for all $n  \ge 1$,
$$
\P(\Lambda_{3n,n} \text{ is $c_p$-crossable}) \ge 1-e^{-cn^{\alpha}}
$$
\end{thm}
In the rest of the section, we fix $p>p_c$. We also denote by $B(x,r)$ the graph distance balls in $\cC_\infty$.
\cref{thm:RSW_perc} will be a quick application of a combination of results in Barlow \cite{Barlow_perc} and \cref{lem:RSW_condition}. Let us begin with a standard lemma:
\begin{lemma}\label{lem:inclusion}
There exist constants $C_{\Euc} := C_{\Euc}(p), c>0$ such that for all $n \ge 1$,
$$\P_p (B(o,C_{\Euc}^{-1}n) \subseteq  \Lambda_n \subset B(o,C_{\Euc} n)) \ge 1-e^{-cn}.$$
where recall that $o$ is the vertex of $\cC_\infty$ nearest to the origin.
\end{lemma}
\begin{proof}
It is easy to see by triangle inequality and the fact that graph distance in $\cC_\infty$ is bigger than that in $\Z^2$, that for any point $x \in B(o,C_{\Euc}^{-1}n$), $|x|\le |o|+ C_{\Euc}^{-1}n$ where $|\cdot|$ is the $\ell_1$-norm. On the other hand, by \cite[eq.\ (4) and references therein]{Garet_chemical}, $|o| \le n/2$ with exponentially high probability in $n$. Thus the first inclusion is satisfied for a large enough choice of $C_{\Euc}$ with exponentially high probability in $n$.  The other inclusion is a similar standard application of \cite[Theorem 1.1]{AP_chemical}.
\end{proof}
\begin{lemma}\label{lem:HandQ}
There exist positive constants $C_P,C_V,C_W,c,c',\alpha,d$ such that for all $n \ge 1$, the following events hold with probability at least $1-ce^{-c'n^{\alpha}}$.
\begin{itemize}
\item $|B_2| \ge dn^2$ where $B_2$ is the square as in \cref{lem:RSW_condition} with $z=(0,0)$.
\item $\Lambda_{n \log n}$ is $C_{\Euc},C_P, C_V, C_W$-very good with $N_B \le n^{1/4}$ with $C_{\Euc}$ as in \cref{lem:inclusion}.
\end{itemize}
\end{lemma}
\begin{proof}

We apply \cref{lem:inclusion} to first obtain a constant $C_{\Euc}$ as required. By translation invariance, we can also assume $C_{\Euc} $ is large enough so that $B(b,\Ce^{-1} n) \subset B_2$ where $b$ is the closest point to $(n/2,0)$ (i.e. the center of $B_2$ with $z=0$). Thus we assume $B(o,\Ce^{-1}n \log n ) \subset \Lambda_{n \log n}$ and $B(b, \Ce^{-1}n) \subset B_2$ for the rest of the proof admitting  a cost exponentially small in $n$.

 We now show that $B(o, C_{\Euc}n\log n)$ is $C_P,C_V,C_W$ very good with stretched exponentially high probability and $B(b,\Ce^{-1} n)$.
This is essentially a combination of \cite[Theorem 2.18 and Lemma 2.19]{Barlow_perc}, let us provide a brief explanation of the results there. In \cite[Theorem 2.18]{Barlow_perc}, it is proved for a box $Q$ of any size, if certain events $H(Q, \alpha) $ and $D(Q, \alpha)$ hold, then the items in this lemma are satisfied. (The events $H(Q, \alpha) $ and $D(Q, \alpha)$ are certain geometric conditions whose exact definitions will not be important for us.) Later in  \cite[Lemma 2.19]{Barlow_perc}, it is shown that $H(Q, \alpha) $ and $D(Q, \alpha)$ hold on the box $Q$ with stretched exponentially high probability in the size of $Q$. 

To be more precise, we apply \cite[Theorem 2.18]{Barlow_perc} for $Q = \Lambda_{C_{\Euc} Cn\log n}$ for $C = 3C_{\Euc}/2$ and $\alpha = 1/8$. With this choice, the first item in this lemma holds with stretched exponentially high probability in $n$ as per item (a) of  \cite[Theorem 2.18]{Barlow_perc} (with $r = \Ce^{-1}n$, $y=b$) since $|B_2 | \ge |B_{\Ce^{-1}}n|$. Also, the second item holds with stretched exponentially high probability in $n$ as per item (c) of \cite[Theorem 2.18]{Barlow_perc} (with $R  = C_{\Euc}n\log n$ and $y=o$). This finishes the proof.
\end{proof}
\begin{proof}[Proof of \cref{thm:RSW_perc}]
We apply \cref{lem:RSW_condition}. \cref{lem:inclusion} justifies the existence of $C_{\Euc}\ge 1$ with $B(o,C_{Euc}^{-1}n) \subseteq  B(o,n) \subseteq \Lambda_n \subset B(o,C_{\Euc} n)$ losing a probability exponentially small in $n$. Furthermore, \cref{lem:HandQ} justifies the requirement the very good condition and the lower bound on the volume of $B_2$ only losing a probability which is stretched exponentially small in $n$. The lower bound on the distance between the inner rectangle and the outer one is trivial since distances only increase in $\cC_\infty$ (choosing $c_0 = 0.1$ suffices). An application of union bound on the above estimates finish the proof.
\end{proof}

\section{RSW for Delaunay triangulation}\label{sec:voronoi}
Let $\Pi$ be a Poisson point process in $\R^2$ with intensity 1. Recall that a Voronoi cell of $x \in \Pi$ is the set of points in $\R^2$ whose closest point (in Euclidean distance) in $\Pi$ is $x$.  Let $\mathbb T$ denote the Voronoi triangulation which is formed by joining two points in $\Pi$ by a straight line if their cells share a common edge (it is a standard fact that this graph is a.s. a triangulation). We will denote by $d_{\mathbb T}$ the graph distance in $\mathbb T$ and the graph distance ball of radius $r$ around $x$ in $\mathbb T$ is denoted by $B_{\mathbb T}(x,r)$ (for a point $x \in \R^2$, $B(x,r)$ denotes the ball of radius $r$ around a vertex in $\mathbb T$ closest to $x$). We will usually drop the subscript for notational convenience when the graph in question is unambiguous.

 In this section we will prove the following theorem.
\begin{thm}\label{thm:RSW_Poisson}
Let $\mathbb T$ be the Voronoi triangulation formed by  a Poisson process of intensity 1. There exist constants $\crsw, c,\alpha \in (0,1]$ such that for all $n  \ge 1$,
$$
\P(\Lambda_{3n,n} \text{ is $\crsw$-crossable}) \ge 1-e^{-cn^{\alpha}}
$$
\end{thm}

%

We refer to \cite{rousselle15} for some results in this direction, but we failed find a reference to the quantitative nature of the estimates we need, hence we prove it in details.

Fix $s>0$ (think of $s$ as large but constant). Take the lattice $s \Z^2$. For $x\in s \Z^2$,  divide the box $\Lambda_s(x)$ into 400 equal sized boxes of size $s/10$, and call them the \emph{smaller boxes}. Call a box $\Lambda_s(x)$ \textbf{$A$-red} if each of the smaller boxes contain at least one and at most $As^2$ many points in $\Pi$. Clearly, for every $\ve>0$, one can choose a large $s$ and $A$ so that $x$ is red with probability at least $1-\ve$ (since the number of points in the box $\sim$ Poisson $(4s^2)$). We call a box simply \text{red} if we let $A =\infty$ (i.e., we do not specify any upper bound of the number of vertices). We can think of the collection of red boxes of the form $\{\Lambda_s(x)\}_{x \in s\Z^2}$ as a site percolation configuration in $\{0,1\}^{s\Z^2}$ with a vertex being 1 if and only if it is red. Call this percolation configuration $\xi  = \{\xi_x\}_{x \in s\Z^2} \in \{0,1\}^{s\Z^2}$. Observe that $\xi$ is not necessarily an i.i.d.\ Bernoulli as the boxes corresponding to two adjacent vertices in $s\Z^2$ overlap. However, $\xi$ is a $2$-dependant percolation: as soon as graph distance between $x$ and $y$ is strictly greater than 2, $\Lambda_s(x) \cap \Lambda_s(y) = \emptyset$ and consequently $\xi_x$ is independent of $\xi_y$.  

Recall the notion of stochastic domination: $\xi$ stochastically dominates $\xi'$ if one can couple them in the same probability space with $\xi_x \ge \xi'_x$ for all $x\in s\Z^2$.
\begin{lemma}\label{lem:xi}
Fix $\ve>0$ and let $s,\xi$ be chosen as above. Then $\xi$ stochastically dominates a Bernoulli site percolation in $s\Z^2$ with parameter $1-\ve'(\ve)$ with $\ve' \to 0$ as $\ve \to 0$.
\end{lemma}

\begin{proof}
As observed above, $\xi$ is a $2$-dependent site percolation.  So by a result of Liggett, Schonmann and Stacey \cite[Theorem 0.0]{liggett1997domination},  $\xi$ dominates a Bernoulli $(1-\ve')$ site percolation $\xi'$ with $\ve' \to 0$ as $\ve \to 0$.
\end{proof}
A path in a graph is a sequence of vertices $v_0, v_1, \ldots, v_k$ such that $v_i$ is adjacent to $v_{i+1}$ for all $0 \le i \le k$, and the edges $(v_i,v_{i+1})$ for all $0 \le i \le k$.
\begin{lemma}\label{lem:cell}
Let $x,y$ be adjacent vertices in $s \Z^2$ and assume $\Lambda_s(x)$ and $\Lambda_s(y)$ are both $A$-red. Take any two smaller boxes of side length $s/10$ in the rectangle $\cR:= \Lambda_s(x) \cup \Lambda_s(y)$ which are at least Euclidean distance $s/4$ from the boundary of the rectangle. Then for any two points of $\Pi$ in these boxes (which one can always find by definition of red), there exists a path in $\mathbb T $ with at most $L:=800As^2$ many vertices joining them, which lies completely inside $\cR$.
\end{lemma}
\begin{proof}
Join any two points in the first box and the second box by a straight line $\cL$. We claim that the Voronoi cells $\cL$ intersects can only belong to points of $\Pi \cap \cR$. Indeed, if a cell corresponding to a point outside $\cR$ intersects $\cL$, then there is a point on $\cL$ which is closer to a point outside $\cR$ than any point inside $\cR$. Thus the disk of radius at least $s/4$ from this point is empty. This means that one of the smaller boxes in $\cR$ is empty, which is impossible since both the boxes $\Lambda_s(x)$ and $\Lambda_s(y)$ are red. This completes the proof of the claim. It is easy to construct a path in the Voronoi triangulation using only the vertices of the cells $\cL$ intersects. Also since $\cR$ is convex, all the edges of this path also lie inside $\cR$. This path can have at most $800As^2$ many vertices as $\Lambda_s(x)$ and $\Lambda_s(y)$ are $A$-red and hence contains at most $800As^2$ points in total.
\end{proof}
We now state a quick result for Bernoulli site percolation in a graph. An open vertex cluster is a connected component in the graph induced by the open vertices.
Recall that in a percolation configuration, the \emph{chemical distance} between two vertices in the same open vertex cluster is the length shortest path in the cluster connecting those two points. It is well-known (see e.g. Grimmett \cite{grimmett1999percolation}) that there exists a unique open infinite cluster in $\Z^2$ for $p>p_c \approx 0.59$.  
\begin{lemma}\label{lem:Bernoulli_perc}
There exist constants $c,c',C>0$ such that for all $n, k\ge 1$ the following holds.
Take a Bernoulli $p$-site percolation in $\Z^2$ with $p>0.9$ and let $\cC_\infty$ be the unique infinite open vertex cluster. Let $\cC_n=\cC_\infty \cap \Lambda_n$. Let $D_{n,k}$ be the following event
\begin{itemize}
\item The chemical distance diameter of $\cC_n $ is at least $n/2$ and at most $Cn$. 
\item For any vertex in $x \in \Lambda_n \cap \Z^2$, $\Lambda_k(x) \cap \cC_\infty \neq \emptyset$.
\item $\max\{|\cH(x)|: x \in \Lambda_n \cap \Z^2\} \le k^2$ where $\cH(x)$ denote the maximal $*$-connected cluster in containing $x$ containing no vertex from $\cC_n$ (call $\cH(x)$ the \textbf{hole} containing $x$)\footnote{$x,y$ are $*$-connected by closed vertices, if there is a path of closed vertices with two consecutive vertices at distance either 1 or 2 in $\Z^2$ (i.e. diagonally adjacent or regular adjacent) connecting $x$ and $y$. It is well-known that a regular cluster is blocked by a $*$-connected circuit.}. 
\end{itemize}
 Then the probability of $D_{n,k}$ is at least $1-cn^4e^{-c'\sqrt{n}} -cn^2e^{-c'k} $.
\end{lemma}
\begin{proof}
This lemma follows from some known results which we first gather. Let $\P_p$ denote the probability measure induced by the percolation and let $D$ denote the chemical distance. Let $x \leftrightarrow y$ denote the event that $x$ is connected to $y$, with $y =\infty$ meaning that $x$ is in an infinite cluster. Let $|x|$ denote the graph distance in $\Z^2$.
\cite[Theorem 1.1]{AP_chemical} states that there is a constant $C = C(p)$ and $c$ such that for all $x \in \Z^2$,
\begin{equation}\label{eq:AP_1}
\P_p(0 \leftrightarrow x, D(0,x) > C|x|) \le e^{-c|x|}.
\end{equation}
Although the result in \cite{AP_chemical} is about bond percolation, it can be easily extended to an analogous result for site percolation since we took $p$ large enough. (e.g. by using \cite[Theorem 0.0]{liggett1997domination} again).
It also follows from \cite[eq.\ (4) and references therein]{Garet_chemical} and translation invariance that for any $x \in \Z^2$,
\begin{equation}
\P_p(\cC_\infty \cap \Lambda_{k}(x) = \emptyset) \le e^{-ck} \label{eq:AP_3}
\end{equation}
Note again, that the results cited hold for bond percolation, but they can be easily translated to site percolation as we took $p$ large enough.

Using \eqref{eq:AP_3}, we see that $\cC_\infty$ intersects $\Lambda_{n/2}$ with a probability which is exponentially high in $n$. This immediately implies that the chemical distance diameter of $\cC_n$ is at least $n/2$ on this event (since $\cC_n$ must intersect the complement of $\Lambda_n$).
Also, \eqref{eq:AP_3} and a union bound over all $x\in \Lambda_n$  ensures the second item is valid with probability at least $1-4n^2e^{-c'k}$. 

We now show that the third item holds with probability $1-cn^2e^{-ck}$. Indeed by isoperimetry of $\Z^2$, if the volume of the $*$-connected hole $\cH(x)$ is bigger than $k^2$ then there is a $*$-connected closed circuit separating $\cC_\infty$ from $x$ of diameter at least $ck$. It is known that this event has probability exponentially small in $k$. An union bound over all the vertices in $\Lambda_n \cap \Z^2$ upper bounds the probability of $\max\{|\cH(x)|: x \in \Lambda_n \cap \Z^2\} > k^2$ by $4n^2e^{-ck}$.

Furthermore, if the chemical distance diameter of $\cC_n$ is bigger than $Cn$ where $C$ is chosen according to \eqref{eq:AP_1}, then there must be vertices $x,y  \in \Lambda_n  \cap \Z^2$ with $D(x,y) >Cn$ and $x \leftrightarrow y$. For any pair with $\|x-y\| \ge \sqrt{n}/10$ this probability is exponentially small in $\sqrt{n}$ by \cref{eq:AP_1}. On the other hand if $\|x-y\| \le \sqrt{n}/10$ and the chemical distance is larger than $Cn$, then the cluster containing $x$ and $y$ must exit $\Lambda_x(\sqrt{n})$ but are not connected within $\Lambda_x(\sqrt{n})$. This must mean there is a $*$-connected closed cluster of diameter at least $c\sqrt{n}$ which separates these clusters. But this event also has probability exponentially small in $\sqrt{n}$. An union bound over the pairs $x,y$ shows that the probability of the diameter of $\cC_n$ being at least $Cn$ is at most $Cn^4e^{-c\sqrt{n}}$.\footnote{This part of the bound is probably not optimal, but since we will be content with  a stretched exponential bound anyway in the end, we do not pursue to make this optimal.}

The result follows by observing that $D_{n,k}$ contains the intersection of these events.
\end{proof}
Recall the notation $B(x,r)$ which denotes the graph distance ball of radius $r$ in the Voronoi triangulation from a point in $\Pi$ which is closest to $x$.

\begin{lemma}\label{lem:vol_upper}
There exist constants $C_{\Euc}, c,c'>0$ such that for all $n \ge 1$, $\Lambda_n \subseteq B(o,C_{\Euc} n)$ with probability at least $1-ce^{-c'\sqrt{n}}$ 
\end{lemma}
\begin{proof}
Choose $A,s $ such that the $2$-dependant percolation $\xi$ as described in \cref{lem:xi} dominates a Bernoulli $(1-\ve')$ site percolation $\xi'$ with $\ve' <0.01$. Couple $\xi, \xi'$ so that $\xi$ dominates $\xi'$. Let $\cC_\infty$ be the unique infinite cluster of $\xi'$ and  let $\cC_{2n} = \Lambda_{2n}\cap \cC_\infty$. Choose $C$ as in \cref{lem:Bernoulli_perc}.

We will pick a box of the form $S:=\Lambda_s(x) \subset \Lambda_n$ with $x \in s\Z^2$ and show that any point in $\Pi \cap S$ can be connected to $o$ by a path of length $O(n)$ in $\mathbb T$ with stretched exponentially high probability. Also assume that if $S$ is empty, this event is vacuously satisfied. Since there are at most $n^2/s^2$ many such boxes, a further union bound does the job.

Assume $|S \cap \Pi| \neq \emptyset$ and assume $\cD'_n:= \cD_{2n, \sqrt{n}/10}$ occurs for $\xi'$ where the event is as described in \cref{lem:Bernoulli_perc} and fix a sample of the Point process in $\cD'_n$. 
Let $\cH_0$ and $\cH_1$ denote the holes of $\xi$ containing the box $S':=\Lambda_s(0) $ and $S$ respectively, where holes are as defined in third item of \cref{lem:Bernoulli_perc}. Note that $|\cH_0| \le n/100$ and $|\cH_1| \le n/100$ on $\cD'_n$ since $\xi$ dominates $\xi'$. These clusters are surrounded by open circuits in $\xi$ each lying completely in $\cC_{2n}$. Now using \cref{lem:cell} by concatenating paths in $\mathbb T$ in the boxes corresponding to these circuits, we can find circuits $C,C'$ in $\mathbb T$ completely surrounding $S$ and $S'$ respectively. This allows us to find a path from any point inside  $C$ to any point inside $C'$ as follows. Find the shortest path in $\mathbb T$ until $C,C'$ is hit and suppose they hit the circuits at vertices $u,v \in \mathbb T$, in boxes $B_u, B_v$ respectively. Clearly, the length of these paths can be at most the volume of the holes, so at most $n/100$ each. Then find the shortest path in $s\Z^2$ using $\cC_{2n}$ joining $B_u$ and $B_v$. This path has length at most $2Cn$ on $\cD'_n$. Furthermore since all the boxes corresponding to vertices in this path are $A$-red, using \cref{lem:cell} we can find a path in $\mathbb T$ joining $u$ and $v$ of length at most $1600As^2Cn$. Thus the length of the path is at most $C'n$ with $C' = 1600As^2C + 1/50$ on $\cD'_n$. We finish by applying \cref{lem:Bernoulli_perc} to lower bound  the probability of $\cD'_n$ by $1-ce^{-c'{\sqrt{n}}} $ for appropriate choices of $c,c'$.
\end{proof}

%
%
%
%
%
%

Now recall the following standard fact about Binomial random variables.
\begin{lemma}\label{lem:bin}
Let $X \sim $ Binomial $(n,1-\ve)$ and fix $c \in (0,1)$. Then for all $n \ge 1$
$$
\P(X < c n) \le e^{-C(\ve, c)n}
$$
where $C(\ve, c) \to \infty$ as $\ve \to 0$.
\end{lemma}

Now we prove a lemma which states a quantitative bound on the probability that a path in the Voronoi triangulation has many long edges. For connected subgraph $S$ in the Voronoi triangulation, its Euclidean diameter is denoted by $\diam_{\Euc}(S) :=\sup\{|x-y|, x,y \in S\}$.

\begin{lemma}\label{lem:small_diameter}
There exist constants $C_0, c_0>0$ such that for all $n \ge 1$ and for all $C>C_0$, the probability that there exists a connected set intersecting $\Lambda_1$ with at most $n$ vertices and Euclidean diameter \footnote{Euclidean diameter of a set $A \subset \R^2$ is $\sup \{|x-y|:x,y \in A\}$. } at least $Cn$ is at most $2^{-Cc_0n}$.
\end{lemma}
\begin{proof}
Fix $\ve>0$ and choose $s$ large enough such that the site percolation $\xi$ defined in \cref{lem:xi} dominates a Bernoulli percolation with high probability, as asserted there. Take a connected subgraph $D$ with at most $n$ vertices intersecting $\Lambda_1$. Let $S(D)$ be the collection of squares of the form $\Lambda_s(x)$ with $x \in s\Z^2$ which intersect (some vertex, edge or triangle of) $D$. Let $k$ be the number of squares in $S(D)$. Because of the large Euclidean diameter  of $D$, $k>\frac{Cn}{2s}$. We say a square is \emph{good} if it is red and all the squares intersected by it is red, otherwise we say it's bad (recall the definition of a red box from the beginning of this section). This corresponds to a 4-dependant site percolation in $s\Z^2$. Thus again using  \cite[Theorem 0.0]{liggett1997domination}, and increasing $s$ if necessary, we can ensure that the collection of good sites dominates a $(1-\ve')$-Bernoulli site percolation in $s\Z^2$ with $\ve'<\ve$. 

We claim that if a square $A$ in $S(D)$ is good, the collection of squares $A'$ of the form $\{\Lambda_s(x): x \in s\Z^2\}$ which intersect $A$ contain at least 1 vertex from $D$. To see this observe that if no edge intersects $A$ then $A \cap \Pi$ is empty, which is not possible as $A$ is red. On the other hand if an edge $e$ intersects $A$, and one of the endpoints of this edge does not lie in one of $A'$, then $e$ must have length at least $2s$. We know that one of the semi-discs of the disc with diameter $e$ is empty in a Voronoi triangulation. This must mean that one of the squares of side $s/10$ in one of the squares $A'$ must be empty, which is a contradiction to the fact that all squares intersecting $A$ are red.  Thus one of the endpoints of $e$ must be in some $A'$. One consequence of this is that the number of good squares is at most $9n $ (since we can overcount a vertex at most 9 times).

It is well known that the number of connected sets in $\Z^2$ with $k $ vertices containing the origin is at most $\Delta^k$ for some $\Delta>0$.
Choose $\ve$ small enough so that $e^{-C(\ve, 1/2)}<(2\Delta)^{-1} $ where $C(\ve,1/2)$ is as in \cref{lem:bin}. Now choose $C >C_0:=36 s=36s(\ve)$ and $c_0 = (2s)^{-1}$. For a fixed connected set in $s\Z^2$ containing the origin and containing $k$ vertices, the number of good vertices $X$ dominates a Binomial $(k,1-\ve)$. Recall that $k >\frac{Cn}{2s} >18 n$. Thus 
$$
\P(X < 9n) \le \P(X <k/2 ) \le e^{-C(\ve, 1/2)k}<(2\Delta)^{-k}
$$
again using \cref{lem:bin}. Since the number of connected sets is at most $\Delta^k$, an union bound gives that the required probability is at most $$(1/2)^k\le (1/2)^{Cn/2s} = (1/2)^{Cc_0n}$$ as desired. 
\end{proof}
An immediate corollary is the following:

\begin{corollary}\label{cor:ball_euc}
There exists a constant $C,c>0$ such that for all $n \ge 1$, $B(o,n) \subseteq \Lambda_{Cn}$ with probability at least $1-e^{-cn}$.
\end{corollary}
\begin{proof}
Note that the event $o \in \Lambda_{n/10}$ has probability at least $e^{-cn^2}$ since $\Pi$ is a Poisson process. Apply \cref{lem:small_diameter} to any path with $n$ vertices starting from $O(n^2)$ many translates of $\Lambda_1$ inside $\Lambda_{n/10}$.
\end{proof}

We now use \cref{lem:small_diameter} to establish a quantitative isoperimetric inequality. We recall some topological notions first. For a finite connected set of vertices $A \subset \mathbb T$, we can consider the subgraph induced by $A$, which we also call $A$ admitting an abuse of notation. This allows us to consider the faces which has all its incident edges in $A$. Overall, we can think of $A$ as a subset of $\R^2$, by taking the union of all the vertices, edges and faces described as such. We define the \emph{complement} of $A$ to be the complement of the union of the faces and edges of $A$. We say $A$ is \emph{simply connected} if the complement has a unique component (which necessarily is the unbounded component a.s.). If $A$ is not simply connected, it's complement may contain a certain number of finite components and one unique infinite component. Let $\partial_V(A)$ denote the vertex boundary of $A$, which is the collection of vertices of $A$ which has some neighbour outside $A$.

\begin{lemma}\label{lem:iso_1}
There exist constants $\ve_0, c,c', C>0$ such that for all $ n \ge 1, \alpha\in (0,2)$ the following holds. The probability that there exists a connected set $A \subset B(o,n)$ with $k \in [ n^{\alpha}, \frac12n^2]$ vertices but $$\cB(A):= \{|\partial_V A| \le \ve_o\sqrt{k} \} \cup\{ \diam(A) \le C\ve_0\sqrt{k}\}$$ holds is at most $ce^{-c'\sqrt{k}}$. Here $\diam$ denotes the graph distance diameter.
\end{lemma}
\begin{proof}
Roughly, the idea is as follows: if $|\partial_V A|$ is small, then the Euclidean diameter must also be small up to a constant factor by \cref{lem:small_diameter}, and consequently a square of small diameter containing at least as many points of $A$ is unlikely. We now make this idea rigorous by carefully tracking the quantifiers.

First observe that it is enough to prove the bound for sets  $A$ which are simply connected, for otherwise we can simply fill in the finite holes, and this operation decreases boundary size but increases the size of $A$ and also decreases $\diam(A)$. Thus $\partial_V A$ has a single connected component. It is a standard fact that the Euclidean diameter of $A$ is the same as the diameter of its boundary. Thus for any point $x$ in $A$, $\Lambda_{2m}(x) \supseteq A$ where $m = \diam_{\Euc} (\partial_VA)$.

Fix $k \in [n^{\alpha} , \frac12n^2]$. First using \cref{cor:ball_euc} find a $C_1$ such that $B(o,n) \subset \Lambda_{C_1n}$ with probability at least $1-e^{-c_1 n}$.  Suppose $\partial_V A$ intersects $\Lambda_1$ for some $A$ with $|A|=k$ and $|\partial_V(A)| \le \ve_0\sqrt{k}$ where $\ve_0$ is to be fine tuned later. Now pick $c_0, C_0$ as in \cref{lem:small_diameter} and $C>C_0$. The probability that any connected set of size at most $\ve_0 \sqrt{k} $ intersecting $\Lambda_1$ has Euclidean diameter at most $C\ve_0 \sqrt{k}$ is at least $ 1-e^{-Cc_0\ve_0\sqrt{k}}.$ Thus on this event $A \subset \Lambda_{m'}$ with $m' =2C\ve_0 \sqrt{k}$; and in particular, on this event,  $\Lambda_{2m}$ contains at least $k$ vertices. Now pick an $\ve_0 = \ve_0(C)$ small enough so that the the probability that the number of vertices in $\Lambda_{2m}$ is at least $k$ is at most $e^{-c'k}$ for some $c'$. Thus overall for this choice of $\ve_0$, the probability that there exists a set $A$ intersecting $\Lambda_1$ such that $\cB(A)$ holds but $|A| \ge k$ is at most 
$e^{-Cc_0\ve_0\sqrt{k}} + e^{-c'k}$.

Finally, on the event that $B(o,n) \subset \Lambda_{C_1n}$, we can take a union bound of the above replacing $\Lambda_1$ by at most $4C_1^2n^2$ many translates of $\Lambda_1$. This yields that the probability of the event in the lemma is at most
$$
4C_1^2n^2 (e^{-Cc_0\ve_0\sqrt{k}} + e^{-c'k}) + e^{-c_1n}.
$$
 Since $k \ge n^{\alpha}$ with $\alpha \in (0,1)$, we can find $c,c'>0$ such that the above quantity is bounded above by $ce^{-c'\sqrt{k}}$. Taking a further union bound over all integers $k \in [n^{\alpha},\frac12n^2]$, we conclude by modifying the choice of $c,c'$ appropriately.
\end{proof}
Notice that the boundary in \cref{lem:iso_1} considers the vertex boundary in the whole Voronoi triangulation. However, the bound in \eqref{eq:i_A} only counts the boundary edges of $A \subset B(y,r)$ inside $B(y,r)$. In the next lemma, we strengthen \cref{lem:iso_1} to show that even discarding the edges going out of $B(y,r)$, there are many edges left over in the boundary with high probability. To do this, we require the following elementary geometric lemma.

\begin{lemma}\label{lem:geometry}
Let $A \subset B(o,n)$ and suppose $\partial_V(A)$ has a single connected component. Assume that $\partial_V(A) \cap \partial_V(B(o,n)) \neq \emptyset$. Let $\partial_{\mathsf{int}}(A)$ denote the set of vertices in $\partial_V(A)$ with at least one neighbour in $B(o,n) \setminus A$ and assume $|\partial_{\mathsf{int}}(A)| = k$. Then for any $v \in \partial_{\mathsf{int}}(A)$, $A \subset B(v, 2k)$.
\end{lemma}
\begin{proof}
Recall $d$ denotes the graph distance in $\mathbb T$. Notice that for any $v  \in \partial_{\mathsf{int}}$, $d(v, \partial_V(B(o,n))) \le k$ since $\partial_V(A) \cap \partial_V(B(o,n)) \neq \emptyset$ and hence $\partial_{\mathsf{int}}$ and $\partial_V(B(o,n))$ must be at distance $1$. Thus triangle inequality yields $d(o,v) \ge n-k$.
Observe that for any vertex $w \in A$, a geodesic from $o$ to $w$ must intersect $\partial_{\mathsf{int}}(A)$ since $\mathbb T$ is planar (since this geodesic must enter $A$ through some vertex, and this vertex is necessarily in $\partial_{\mathsf{int}}(A)$). Let $w'$ be such a vertex which is closest to $o$. By triangle inequality, and since $d(o,w') \ge n-k$, $d(w',w)  =  d(o,w) - d(o,w') \le n - (n-k)  = k$. Then $d(v,w ) \le  d(v,w') + d(w',w)\le 2k $. This completes the proof. 
\end{proof}

\begin{lemma}\label{lem:iso_2}
There exist constants $\ve_1, c,c'>0$ such that for all $ n \ge 1, \alpha\in (0,2)$ and $k \in [n^{\alpha}, \frac12 n^2]$ the following holds. The probability that there exists a connected set $A \subset B(o,n)$ with $k$ vertices but with $|\partial_{\mathsf{int}}(A)| \le \ve_1\sqrt{k}$ is at most $ce^{-c'\sqrt{k}}$ where $\partial_{\mathsf{int}}(A)$ is defined as in \cref{lem:geometry}.
\end{lemma}
\begin{proof}
If all the vertices of $A$ are in $B(o,n-1)$, then $\partial_V(A) = \partial_{\mathsf{int}}(A)$, and we simply choose $\ve_1 = \ve_0$ where $\ve_0$ is as in \cref{lem:iso_1}. On the other hand, if $A $ intersects $\partial B(o,n)$, then by \cref{lem:geometry}, the graph distance diameter of $A$ is at most $4|\partial_{\mathsf{int}}(A)|$.
Thus again by \cref{lem:iso_1}, we can choose $\ve_1 = \min\{\ve_0, C\ve_0/4\} $. 
\end{proof}

\begin{lemma}\label{lem:iso_3}
For any $A \subseteq B(o,n)$ ,
$$
i_{B(o,n)}(A) = \frac{|\partial_E(A ,  B(o,n) \setminus A)|}{|A|_E} \ge  \frac{|\partial_{\mathsf{int}}(A)|}{7|A|}
$$
where $|A|_E$ is the sum of the degrees of vertices in $A$ counting only edges in $B(o,n)$ and $\partial_{\mathsf{int}}(A)$ is as in \cref{lem:geometry}.
\end{lemma}
\begin{proof}
Let $E$ be the edge set of the subgraph induced by $A$. Notice that this subgraph is a subgraph of a triangulation, hence the faces form a collection of triangles and (potentially non-simple) polygons, call the latter outer faces. Without loss of generality, we can assume that there is only one outer face, as otherwise we can fill in the bounded faces, thereby decreasing $i_{B(o,n)}(A) $. Let $|P|$ denote the perimeter of the outer face which counts the number of edges in it, with the edges having both sides adjacent to $P$ counted twice. Let $F$ denote the set of triangles in this graph. By Euler's formula, $|A| -|E| + |F| + 1 = 2$. Also note, $2|E| = 3|F|+|P|$. Combining, we get $|E| = 3|A|-|P|-3 \le 3|A|$. Also note $|A|_E = 2|E| + |\partial_E(A ,  B(o,n) \setminus A))| $. Thus
$$
i_{B(o,n)}(A) \ge \frac{|\partial_E(A ,  B(o,n) \setminus A))|}{6|A|+ |\partial_E(A ,  B(o,n) \setminus A))|} \ge \frac{|\partial_{\mathsf{int}}(A)|}{7|A|}
$$ 
The last inequality follows from $|\partial_E(A ,  B(o,n) \setminus A))| \ge |\partial_{\mathsf{int}}(A)|$, the fact that $x \mapsto x/(6+x)$ is increasing in $x$ and $x/(6+x) \ge x/7$ for all $x \in (0,1)$.
\end{proof}

\begin{proof}[Proof of \cref{thm:RSW_Poisson}]
We will simply prove that the items in \cref{lem:crossing_criterion} holds with stretched exponentially high probability for appropriate choice of constants (we use the notations there). Firstly, $(iv)$ holds with exponentially high probability in $n^2$ for a small enough choice of $d$ using standard estimate of a Poisson variable.  It follows from \cref{lem:vol_upper,cor:ball_euc} that $(i)$ holds with stretched exponentially high probability in $n \log n$ for an appropriately large choice of $C_{\Euc}$. Now fix $r \in [n^{1/9}, C_{\Euc}n \log n]$. It is easy to see that the volume of $B(o,r)$ is upper and lower bounded by some constant in $r^2$ with stretched exponentially high probability in $r$, again using \cref{lem:vol_upper,cor:ball_euc} and standard properties of a Poisson process. Applying \cref{cor:ball_euc} to $O(n^2)$
many translates of $\Lambda_1$, we can also ensure $(ii),(v)$ holds with exponentially high probability in $r$. Finally, choosing $\ve_1$ as in \cref{lem:iso_2}, we can ensure that $(iii)$ holds with $C_I = \ve_1/7$ and with probability at least stretched exponentially high in $r$ (and consequently stretched exponentially high in $n$.). Now we take an union bound over integers $r \in [n^{1/9}, C_{\Euc}n \log n]$ to complete the proof.
\end{proof}

\section{General criterion for macroscopic decorrelation in Uniform spanning trees in random environment}\label{sec:UST}
Let $\mu$ be a probability measure supported on infinite, locally finite, one ended, random, planar graphs embedded in a proper way in the plane. Recall that an embedding is proper if no two edges cross each other. In this section, we present a result which is an adaptation of \cite[Theorem 4.21]{BLR16}, but for random graphs with law $\mu$. Recall that the examples which concern us are Poisson Voronoi triangulation, and the infinite cluster of a supercritical Bernoulli percolation. We now state the two main assumptions on $\mu$. Let $G = (V(G), E(G))$ be a sample from $\mu$. Recall the definition of $C$-crossable from \cref{def:RSW}.

\begin{enumerate}[{(}i{)}]
\item The law of $\mu$ is invariant under translations and $\pi/2$-rotations of the plane.
\item (RSW). There exist constants $c_\mu, d,d', \alpha>0 $ such that for all $n \ge 1$, $\Lambda_{4n,n}$ is $c_\mu$-crossable with $\mu$-probability at least $1-de^{-d'n^{\alpha}}$.
\item (Vol) There exist constants $C_\mu, \beta>0$ such that  $\mu(|\Lambda_n  \cap V(G)| \ge C_\mu n^2) \le e^{-\beta n^2}$.
\end{enumerate}
We now state our result for decoupling of uniform spanning trees from the point of view of scaling limits. Let $\delta G$ be a rescaling of the embedded graph $G$ by $\delta$. Let $D\subset \R^2$ be a domain which will always be an open, simply connected set in this section. Let $D^\delta$ denote the graph induced by the vertices of $G^\delta$ in $D$ where $\partial_V D^\delta$ is identified into a single vertex (i.e. we consider a wired boundary condition). A spanning tree of a finite graph $H$ is a subgraph which contains all the vertices of $H$ and does not contain any cycle. Recall that a uniform spanning tree on a finite graph is simply a uniformly picked spanning tree of the graph.

Random walks and uniform spanning trees are intimately related to each other via the celebrated \textbf{Wilson's algorithm} which we quickly describe here. Order the vertices of $D^\delta$ in any order. Now perform a loop erased random walk, i.e. a simple random walk where one chronologically erases the loop. We continue this until the walk hits the boundary vertex (the wired boundary). This samples a simple path $\gamma$ starting at $v_1$ and ending at the boundary vertex. Then we sample the next vertex in the ordering which is not in $\gamma$, and repeat the same procedure with the new boundary being the boundary vertex union $\gamma$. We iterate, until all the vertex belong to some path. The final object thus obtained is a sample from a uniform spanning tree. See \cite{wilson96}, or \cite[Section 4.1]{LyonsPeres} for a proof of this fact.

\begin{thm}\label{thm:coupling}
Suppose $G$ sampled from $\mu$ satisfies the above conditions for some constants $c_\mu, \alpha, C_\mu, \beta$ and let $z_1,\ldots, z_k$ be $k$ points in a domain $D$ where $  \Lambda_1 \subset D \subset \Lambda_{10}$ and let $r = \min_{i \neq j} |z_i - z_j| \wedge d(z_i, \partial D)$ where $d$ is the Euclidean distance.
There exists a constant $c = c(c_\mu, \alpha, C_\mu, \beta)>0$ such that for all $\ve, \ve'>0$, there exists a $\delta_0 = \delta_0(\ve')$ such that for all $\delta \le \ve \wedge \delta_0(\ve')$ the following holds.  There exists a collection of graphs $\cG$ with $\mu(\cG) >1-\ve$ such that for any $G$ in $\cG$ the following holds.

Let $\cT^\delta$ be a sample of a wired Uniform spanning tree in $D^\delta$. Then there exists a coupling $\mathbf P^G$ between $\cT^\delta$ and a collection $\{\cT_i^\delta\}_{1 \le i \le k}$ such that
\begin{itemize}
\item $\{\cT_i^\delta\}_{1 \le i \le k}$ are i.i.d. copies of $\cT^\delta$.
\item $\cT_i^\delta \cap \Lambda_R(z_i)  =  \cT^\delta \cap\Lambda_R(z_i)$ for all $1 \le i \le k$ where $R$ is a random variable satisfying
$$
\mathbf P^G (R \le \ve' r) \le (\ve')^c.
$$
\end{itemize}
\end{thm}

\begin{remark}\label{a.s.}
If we fix $\ve'>0$ and apply \cref{thm:coupling} for a sequence $\delta_k = \ve_k = 2^{-k}$ with $2^{-k} < \delta_0(\ve')$, then by Borel--Cantelli, there exists a collection $\cG$ with $\mu(\cG)=1$ such that for any $G \in \cG$, the coupling $\mathbf P^G$ as in \cref{thm:coupling} holds for all $k$ large enough depending on $G$.
\end{remark}

In \cite[Theorem 4.21]{BLR16} an analogous version was proved but for a \emph{fixed} graph, where the condition (RSW) was valid above a certain fixed scale (called $\delta_0$ in that article). The main new input in \cref{thm:coupling} is that an analogous result holds with high probability with the more general condition (RSW) above. One can also get an almost sure statement if we allow ourselves to choose a subsequence of $\delta$ (notice the dependance $\ve \ge \delta$ in the statement of the theorem above), see \cref{a.s.} below

Let $A(z, m, n)$ be the annulus $\Lambda_n \setminus \Lambda_m$ for $n >m$. Let $R_1,\ldots, R_4$ be the translations and 90 degree rotations of $\Lambda_{2n,(n-m)/2}$ whose union is $A(z,m,n)$. We say $A(z, m, n)$ is $c^4$-crossable if all the rectangles $R_1, \ldots, R_4$ are $c$-crossable (by Markov property of random walk, the probability for a random walk to make a full turn in $A(z, m, n)$ is at least $c^4$).
Let $$R^\delta(z) = \max\{2^i \delta: A(z, 2^i\delta, 2^{i+1}\delta) \text{ is \textbf{not} $c_\mu$-crossable in $\delta$G}\}.$$ Define 

$$R_{\max}^\delta = \max\{R^\delta(z) : z \in \Lambda_{10} \cap \delta \Z^2  \}.$$
\begin{lemma}\label{lem:hole}
There exists a constant $C>0$ such that for all $\ve, \delta>0$, 
$$
\mu(R_{\max}^\delta >R_0^\delta) \le \ve.
$$
where
\begin{equation}
R_0^\delta = R_0^\delta( \ve)= \delta \left[\ln\left(\frac{C}{\ve \delta^2}\right)    \right]^{1/\alpha}\label{eq:R0}
\end{equation}
and $\alpha$ is as in (RSW).
\end{lemma}
Note that for  any $\ve$ which is at least $\delta^m$ for some $m$, $R^\delta \to 0$ as $\delta \to 0$.
\begin{proof}
Let $$\cB_1  = \cup_{ j \ge m} \{A(0, 2^{j}\delta, 2^{j+1}\delta) \text{ is not $c_\mu$-crossable in $\delta G$}\}  . $$ 
Notice that by (RSW), the invariance of $\mu$ under translation and $\pi/2$-rotations, and a union bound,
$$
\mu(\cB_1) \le \sum_{ j \ge m} 4e^{-2^{\alpha j}} = 4\sum_{j \ge 0}(e^{-2^{\alpha m}})^{2^{\alpha j}} \le C' e^{-2^{\alpha m}}.
$$
for some constant $C'>0$ independent of everything else. By translation invariance, the same bound is true if we replace $0$ by any other $z \in \delta \Z^2$. Since there are at most $400/\delta^2$ many points in $\Lambda_{10} \cap \delta \Z^2$, by an union bound:
$$
\mu(R_{\max}^\delta >R_0^\delta)  \le C' \frac{400}{\delta^2 } e^{-2^{\alpha m}}
$$
where we choose $2^m\delta = R_0^\delta$. The right hand side above is at most $\ve$ if we choose $C = 400C'$. This completes the proof.
\end{proof}
\begin{remark}\label{a.s.2}
Note that for a choice of the sequence $\delta_k = \ve_k = 2^{-k}$ by Borel--Cantelli, $\mu$-a.s. $R_{\max}^{\delta_k} \le R_0^{\delta_k}$ for all $k$ large enough. Also for this choice, $R_0^{\delta_k} \to 0$ as $k \to \infty$.
\end{remark}

An application of \cref{lem:hole} is that for a large enough rectangle depending on $R_{\max}^\delta$, RSW holds. 
\begin{lemma}\label{lem:RSW}
Fix a graph $G$ such that $ R_{\max}^\delta \le R_0^\delta$ where $R_0^\delta$ is as in \eqref{eq:R0}. Then for any $m \ge R_0$ and any rectangle  lying completely inside $D$ which is a translate of $\Lambda_{20m, 5m}$ is $c_\mu^{10}$-crossable. The same holds true for any rectangle which is a translation and a $\pi/2$-rotation of  $\Lambda_{20m, 5m}$ and lies completely inside $D$.
\end{lemma}
\begin{proof}
This is a standard consequence of RSW theory. Indeed, any rectangle of the form specified by the lemma can be covered by at most 10 many $2m \times 4m$ rectangles which is one of the rectangles of some $A(z, 2m, 4m)$ with $z \in \delta \Z^2 \cap \Lambda_{10}$ (recall that $D \subset \Lambda_1$). The rest follows by applying the Markov property of the random walk and using the fact that $m \ge R_0 \ge R^\delta_{\max}$. 
\end{proof}
One standard application of \cref{lem:RSW} is a Beurling type hitting estimate for a random walk:
\begin{lemma}\label{lem:Beurling}
There exists $c,c'$ such that for all $\ve>0$ and for all $\delta \le \ve$, the following holds.
Let $K \subset D$ be a connected set. Fix a $G \sim \mu$ such that $R^\delta_{\max} \le R_0^\delta(\ve)$ and let $\mathbf P^G_v$ be the law of a simple random walk $X$ in $G$ started from $v$. Let $d(v, \partial D)$ (resp. $d(v,K)$) be the Euclidean distance between $v$ and $\partial D$ (resp. $K$). Then
\begin{equation*}
\mathbf P^G_v(X \text{ exits $\Lambda_{d(v, \partial D)}(v)$ before hitting $K$}) \le c \left(\frac{d(v,K) \vee R_0^\delta(\ve)}{d(v,  \partial D)}\right)^{c'} 
\end{equation*}
\end{lemma}
\begin{proof}
This is standard once we have \cref{lem:RSW}, so we skip the proof. We point out that the term $R_0^\delta(\ve)$ in the numerator appears because we can apply RSW once the scale is larger than $cR_0^\delta$ since $R_{\max}^\delta \le R_0^\delta$ by the choice of $G$.
\end{proof}

Armed with this estimate, the rest of the proof of \cref{thm:coupling} follows the same line of argument as in \cite{BLR16}. We provide a sketch of the argument pointing out the crucial differences. We now describe the \emph{good algorithm} from \cite[Lemma 4.18]{BLR16}. Let $z \in D$ and suppose $r>0$ is small enough so that $\Lambda_{2r}(z) \subset D$. Fix a $G \sim \mu$ and we now describe a way of sampling the branches of $\cT^\delta$ from the vertices of $\Lambda_r^\delta(z)$. Let $\cQ_j$ be the collection of vertices of $G$ which are furthest from $z$ in each cell of $\Lambda^\delta_{(1+2^{-j})r} (z) \cap r6^{-j}\Z^2$, from which a branch is not sampled before (if there is no such vertex, we ignore that cell). At each step $j$, we sample from $\cQ_j$ in any order. This results in a tree $\cT^\delta_j$ which is the union of all the branches sampled in steps 1 up to $j$. We continue until we exhaust all the vertices in $\Lambda^\delta_r$.
\begin{lemma}\label{lem:schramm_finiteness}
Fix $D,z,r$ as above. For all $\ve, \ve'>0$, there exists a $j_0 = j_0(\ve')$ such that for all $\delta  \le \delta_0(r, \ve') \wedge \ve$  the following holds.
Fix a $G$ such that $R^\delta_{\max} \le R_0^\delta(\ve)$ where $R_0^\delta(\ve)$ is as in \eqref{eq:R0}  and $|\Lambda_{10} \cap \delta G| \le C_\mu100/\delta^2$ (i.e., (Vol) is satisfied). Then with probability at least $1-\ve'$
\begin{enumerate}[{(}i{)}]
\item The random walks emanating from all branches in $\cQ_j$ for $j \ge j_0$ stay in the square $\Lambda_{2r}(z)$.
\item All the branches sampled from vertices in $\cQ_j $ for $j >j_0$ until they hit $\cT_{j_0} \cup \partial D^\delta$ have Euclidean diameter at most $\ve' r$.
\end{enumerate}
\end{lemma}
\begin{proof}
The proof follows an argument similar to \cite[Lemma 4.18]{BLR16}, so we provide a sketch. Fix $j_0 = j_0(\ve')$ which will be fine tuned later. Let
$$
j_{\max} = \log_6\left(4r/R_0^\delta\right).
$$
where $R_0^\delta$ is as in \eqref{eq:R0}.
First, notice that the distance between a vertex in $\cQ_{j-1}$ and another in $\cQ_j$ is at most $4 \cdot 6^{-j} r$. For small $j<j_{\max}$, notice that $4 \cdot 6^{-j}r \ge R_0^\delta$. Thus using the Beurling type estimate of \cref{lem:Beurling}, the probability that the random walk started from a vertex in $\cQ_j$ reaches distance $C_06^{-j}r$ without hitting a sampled branch is at most $1/2$ for a large enough choice of $C_0$. Applying this bound $j^2$ times, and using Markovian property of the walk, the probability that the walker reaches distance $j^26^{-j } r$ from a point $w \in \cQ_j$ without hitting any sampled branch is at most $(1/2)^{j^2/C_0}$. Let $\cD(w,j)$ be the event described above.
Applying this bound over all branches for $j_0 \le j \le j_{\max}$, we see that 
\begin{equation}
\P(\cup_{w \in \cQ_j} \cup_{j_0 \le j \le j_{\max}} \cD(w,j)) \le \sum_{j \ge j_0} 6^{2j} (1/2)^{j^2} \le \ve'/2, \label{eq:E1}
\end{equation}
for a large enough choice of $r$.

On the other hand, for $j \ge j_{\max}$, the Beurling bound kicks in only at distance $4\cdot 6^{-j_{\max}}$. Thus applying this crude bound, the probability that any $w \in \cQ_j$ reaches distance larger than $j^2 6^{-j}r$ without hitting any 	branch in $\cT_j$ is at most $(1/2)^{j_{\max}^2/C_0}$. Call this event $\tilde \cD(w,j)$. Applying the crude bound that the total number of vertices in $\Lambda_{10}$ is at most $100C_\mu/\delta^2$, we see that 

\begin{equation}
\P(\cup_{j \ge j_{\max}}\cup_{w \in \cD_j} \tilde D(w,j)) \le 100C_\mu\delta^{-2}(1/2)^{j_{\max}^2/C_0}  \le \ve'/2 \label{eq:E2}
\end{equation}
for a small enough choice of $\delta = \delta ( \ve',r) \wedge \ve$ (this is the part where the proof differs from that in \cite{BLR16}). Indeed, we can write $j_{\max}$ as
\begin{equation*}
j_{\max}  = \log_6 \left(\frac{4r}{\delta}\right) - \log_6\Big(\log^{1/\alpha} \Big(\frac{C}{\ve \delta^2}\Big)\Big) \ge \log_6 \left(\frac{4r}{\delta}\right) - \log_6\Big(\log^{1/\alpha} \Big(\frac{C}{ \delta^3}\Big)\Big).
\end{equation*}
since $\ve \ge \delta$. In the above expression, $j_{\max} $ is much larger than $-C \log(\delta)$ for some constant $C$  and $\delta$ small enough (depending only on $r$), and thus $(1/2)^{j_{\max}^2}$ is much smaller than any polynomial in $\delta^{-1}$.

Now it is easy to see that on the complement of \eqref{eq:E1} and \eqref{eq:E2}, the diameter of the branches sampled after step $j_0$ is at most
$$
\sum_{j \ge j_0}  j^2 6^{-j}r.
$$
For the same choice of $j_0$, the above quantity is less than $\ve' r$. Thus both items (i) and (ii) are satisfied on the complement of the event, thereby completing the proof.
\end{proof}
\subsection{The coupling.} The coupling of \cref{thm:coupling} is as described in Section 4.4 of \cite{BLR16} and proceeds in two stages. First we couple around a single point. Then if it fails, we iterate until we succeed. The important difference from \cite{BLR16} is that we have to pick a `good' sample from $\mu$ first. Fix $\ve>0$ and pick a graph $G \in \cG$ where $\cG$ is the collection of graphs satisfying satisfying $ R_{\max}^\delta \le R_0$ where $R_0$ is as in \eqref{eq:R0} and $|\Lambda_{10} \cap \delta G| \le C_\mu100/\delta^2$. Applying \cref{lem:hole} and assumption (Vol), we obtain
\begin{equation}
\mu(\cG) \ge 1-\ve-\exp\left(-\frac{\beta}{\delta^2}\right) \label{eq:muG}
\end{equation}

We assume thoughout that we have picked a graph $G \in \cG$ in the rest of the description of the coupling.
\paragraph{Base coupling.} Pick $z\in D$. We will now describe a coupling between a wired UST $\cT$ in $D$ and another wired UST $\tilde \cT$  of $\Lambda_{10}^\delta$ using the following steps, which we call the base coupling. It will be described with respect to a scale $r$ satisfying $\Lambda_{2r}(z) \subset D$. Given a vertex $w$, let $\gamma(w)$ (resp. $\tilde \gamma(w)$) denote the wired UST branch of $\cT$ (resp. $\tilde \cT$) sampled via Wilson's algorithm 
\begin{itemize}
\item Take $w_1 \in A(z, 0.8r, 0.9 r)$ and sample $\gamma(w_1)$ and $\tilde \gamma(w_1)$ independently until they both hit the boundary of their respective domains. Let $E_1$ be the event that both $\gamma(w_1)$ and $\tilde{\gamma}(w_1)$ stay outside $\Lambda_{0.7r}(v)$.
\item Conditional on the event $E_1$ holding, we couple the loop-erased random walk emanating from a vertex  $w_2 \in A(v, 0.3r, 0.4r)$ as follows. We sample a loop-erased random walk until hitting either  $\gamma(w_1)\cup \partial D^{\delta}$ or $\tilde{\gamma}(w_1) \cup \partial \Lambda_{10}^{\delta}$. Without loss of generality, we assume that the walk  intersects $\gamma(w_1)\cup \partial D^{\delta}$ at time $t_1$. Then we continue the random walk from that point until it intersects $\tilde{\gamma}(w_1) \cup \partial \Lambda_{10}^{\delta}$ at time $t_2$ and its path is denoted by $\tilde{\gamma}(w_2)$. Let $E_2$ be the event that $\gamma(w_2)$ and $\tilde{\gamma}(w_2)$ agree in $\Lambda_{0.6r}(v)$.
\item 
Suppose that events $E_1$ and $E_2$ hold. Fix a $j_0=j_0(1/2)$ as defined in \cref{lem:schramm_finiteness}. As the description of good algorithm above, let $\mathcal{Q}_j$ be a set of vertices in $\{0.1r6^{-j}\Z^2\}_{j\geq 0}\cap \Lambda_{0.1r}(v)$ which are chosen that each one is furthest away from $v$ within the small square. Define the event $E_3$ to be the branches emanating from all the vertices in $\cup_{j\leq j_0}\mathcal{Q}_j$ of $\mathcal{T}^{\delta}$ and $\widetilde{\mathcal{T}}^{\delta}$ agree in $\Lambda_{0.5r}(v)$.

\item Assume that events $E_1$, $E_2$ and $E_3$ hold. Let $E_4$ be the event that the remaining branches starting from vertices in $\cup_{j> j_0}\mathcal{Q}_j$ of $\mathcal{T}^{\delta}$ and $\widetilde{\mathcal{T}}^{\delta}$ agree in $\Lambda_{0.1r}(v)$.
\end{itemize}
We will show below that the base coupling succeeds with a uniformly positive probability.
\paragraph{Iteration of the base coupling.} We now want to iterate the above base coupling, decreasing the scale at every step and in the end want to conclude that the coupling succeeds after geometric many tries. Also, we want to conclude that after geometric many tries, there is enough space around $z$ on which the spanning trees are coupled. 

We say a $z$ has \emph{isolation radius} $6^{-k}$ at scale $r$ at any step in the above base coupling if $\Lambda_{6^{-k}r}(z)$ does not intersect any sampled branches and $k$ is the minimal such integer. Fix a large constant $C_0$. We start the iteration with an attempt at the base coupling with scale $r$. If it is successful, we say the iteration is complete and the coupling is successful. If not, let $I_1$ be the isolation radius at scale $r$. If $6^{-I_1}r <C_0R_0^\delta$, we abort the coupling and say that the coupling failed. Otherwise, if the base coupling was not successful and $6^{-I_1}r \ge C_0R_0^\delta$ we attempt another base coupling at scale $6^{-I_1}r/2$ in the domain $\cT \setminus \cT_0^\d$ where $\cT_0^\d$ is the tree sampled in $D^\delta$ in the first step (and ignore the tree sampled for $\Lambda_{10}^\delta$). Iterating this process, we obtain the isolation radii at any step $j$ (if we have not aborted) is $\sum_{1 \le k \le j}I_k$ at scale $r$.

Let $N$ be the smallest $j$ such that we either abort the coupling or the coupling succeeds. Let $I_z = I_1+\ldots+I_N$.

\paragraph{Full coupling.} Pick distinct points $z_1, \ldots, z_k$ in $D$. Let $r>0$ be such that $\Lambda_{2r}(z_i) \subset D$ for all $i$ and $|z_i - z_j| >2r$ for all $i \neq j$. Assume that $\delta$ is small enough so that $0.01r > R_0$ where $R_0$ is as in $\eqref{eq:R0}$. Since we picked a $G$ with $R_{\max} \le R_0$, uniform crossing is possible in $A(z_i, r, 1.1 r)$ for all $1 \le i \le k$. Consequently, there must exist a circuit lying completely in $A(z_i, r, 1.1 r)$. We first perform Wilson's algorithm from all the vertices from these circuits in $A(z_i, r, 1.1 r)$. Let $J_{z_i}$ be the isolation radius at scale $r$ seen from $z_i$ for each $ 1 \le i \le k$. Now perform iteration of base coupling around each vertex independently, starting from scale $6^{-J_{z_i}}r/2$. Let $I_{z_i}$ be the isolation radius as defined in the iteration of the base coupling part of the description. Let $$I = \max \{J_{z_i} + I_{z_i} : 1 \le i \le k\}.$$

Let $\mathbf P^G$ be this coupling, and this will be the coupling used to prove \cref{thm:coupling}.

\begin{prop}\label{prop:tail}
Let $G \in \cG$ as in \eqref{eq:muG}, and let $I, \mathbf P^G$ be as above. Then
there exist constants $c,c'>0$ such that for all $i >0$, and $\delta $ small enough,
$$
\mathbf P^G(I>i) \le ce^{-c'i} + \delta^{c'}.
$$
\end{prop}
\begin{proof}
The proof of this statement follows the exact same lines as \cite[Lemmas 4.19, 4.20 and Theorem 4.21]{BLR16}, replacing the scale $\delta_0$ which was order $1$ there by $R_0/\delta$ which is of order $\log(\delta^{-1})$. This allows us to upper bound $\mathbf P^G(I>i) $ by $ce^{-c'i} + (\delta \log(\delta^{-1}))^{c''}$ which can be easily bounded by $ce^{-c'i} + \delta^{c'}$ for a smaller choice of $c'$.
\end{proof}
\begin{proof}[Proof of \cref{thm:coupling}]
 Choose $\delta_0(\ve')$ small enough so that for all $\delta< \delta_0$, $(\delta \log(\delta^{-1}))^{c''} \le \ve'$ where the expression is as in \cref{prop:tail}. Pick an even smaller $\delta \le \ve$ if needed. Thus by \eqref{eq:muG}, we can pick a $\cG$ with $\mu(\cG) \ge 1-\ve - \exp(-\beta / \ve^2) \ge 1-2\ve$. For such a $G \in \cG$, the choice of $\delta$, we perform the full coupling described above. By \cref{prop:tail}, if we rename $R = 6^{-I} r$, then 
 $$
 \mathbf P^G(R \le \ve' r) \le  \ve' + (\ve')^c \le (\ve')^{c''}.
 $$
 for some small enough $c''$. This completes the proof.
\end{proof}

\section{Dimer model in random graph}\label{sec:dimer}
In this section we outline a scaling limit result for the height function of the dimer model on a class of random graphs with law $\mu$. Apart from the three properties outlined in \cref{sec:UST}, we add the additional fourth assumption of quenched Invariance principle:

\begin{enumerate}
\item[\textit{(iv)}] For $\mu$-almost sure $\Gamma$, the following holds. As $\delta
  \to 0$, the continuous time random walk $\{X_t\}_{t \ge 0}$ on
  $G$
  started from a nearest vertex to $0$ converges to Brownian motion in the following sense:
  $$
 {\delta( \tilde X_{t/\delta^2})_{t \ge 0}} \xrightarrow[\delta \to 0 ]{(d) }
(B_{\phi(t)})_{0 \leq t}
$$
where $(B_t, t \ge 0)$ is a two dimensional standard Brownian motion  started from
 $0$. The convergence above is in law and using the uniform topology on curves up to parametrisation.
\end{enumerate}
\begin{remark}\label{rmk:quenched}
By \cite[Theorem 1.1]{berger2007quenched} and by \cite[Theorem 1.1]{rousselle15}, Quenched invariance principle holds for random walk on the infinite cluster for supercritical percolation on $\Z^2$ and for Poisson Voronoi triangulation in $\R^2$.
\end{remark}

Recall that in a bipartite graph, a \emph{dimer configuration} is perfect matching of the black and the white vertices (every black vertex is matched to exactly one white vertex). The dimer model is a uniform probability measure on all possible perfect matchings on it. Take $D \subset \C$ be a simply connected domain with a smooth boundary $\partial D$.
Let $\Gamma^\delta = \delta G \cap D$ and assume that the boundary vertices form a simple cycle which converges in the Hausdorff sense to $\partial D$ almost surely (crossing estimate ensures that such a cycle exists a.s. at least for all small enough $\delta$ along a subsequence which is at a Hausdorff distance $O(\delta^{1/2})$ from $\partial D$). Take $(\Gamma^\dagger)^\delta$ to be the dual graph. Now introduce white vertices at the points where the primal and the dual edges cross, and let $H^\delta$ be the final graph obtained. Let $\mu_{\dim}$ denote the uniform probabilty measure on $H^\delta$. We refer to \cite{KPWTemperley} or \cite{BLR_riemann1} for more details of this construction.

It is well-known that a dimer configuration is uniquely associated to a \textbf{height function} $h^\delta: F(H^\delta) \to \R$ where $F(H^\delta)$ is the collection of faces of $H^\delta$. Furthermore, the height function can be associated to the winding of the wired Uniform spanning tree on $G^\delta$ as follows (here the boundary cycle is wired). Fix a point $x$ on $D$ and let $x^\delta$ be the point closest to it. Take $z \in D$ and a point $z^\delta$ closest to $z$ in $D$. Let $\gamma^\delta(z)$ denote the path formed by UST branch started from $x$, hitting the boundary cycle and then moving along the boundary cycle to $x^\delta$. Then it is known (see \cite[Section 5]{KPWTemperley} or \cite{BLR_riemann2} for a more detailed treatment) that one can set things up so that the height function $h^\delta(z)$ is the amount of winding done by $\gamma^\delta(z)$. We refer to \cite[Section 2]{BLR16} for a precise definition of this winding (called `intrinsic winding' there). We now extend $h^\delta(z)$ in a natural way to all of $D$ (e.g. by considering Voronoi cells around the center of the faces of $H^\delta$), so that we can integrate $h^\delta$ with a smooth compactly supported test function in $D$.

With this setup, the arguments in \cite[Section 5]{BLR16} can be readily applied in a quenched sense. Indeed by \cref{a.s.,a.s.2}, if we choose $\ve = \delta_k = 2^{-k}$ then $\mu$-a.s. $R_{\max}^{\delta_k} \le R_0^{\delta_k}$  and \cref{thm:coupling} holds for any $\ve'$ if we choose $k$ large enough. Thus we can use the Invariance principle and the convergence of loop erased random walk to SLE$_2$ to control the macroscopic part of the winding (\cite{YY}) and \cref{thm:coupling} to control the microscopic part (basically the independence cancels out the cross terms in a joint moment calculation). We obtain the following theorem whose proof we omit in order to keep the exposition succinct and to avoid repitition.
\begin{thm}\label{thm:GFF}
Suppose $\mu$ satisfies the conditions of \cref{sec:UST} and the Quenched invariance principle (item (iv)) as above. Then there exists a collection $\cG_0$ with $\mu(\cG_0)=1$ such that for any $G \in \cG_0$ 
the following holds. For any compactly supported test function $\varphi $ in $D$, 
$$
\int_D \bar h^{\d_k}(z)\varphi(z)dz \xrightarrow[]{k \to \infty} \int_D \sqrt{2}h_D^{\mathsf{GFF}}(z) \varphi(z)dz.
$$
along the subsequence $\delta_k = 2^{-k}$ in law, where $\bar h^{\d} = h^{\d} - \mathbf{E}^G(h^{\d})$ and $h_D^{\mathsf{GFF}}$ is a Gaussian free field with Dirichlet boundary condition in $D$.
\end{thm}
The above result can also be refined in various directions, we refer to \cite{BLR16} for more details.

We finish the article by stating a corollary of \cref{thm:GFF} for the special cases of random graphs we focussed on in this article, namely: Poisson Voronoi triangulation $\mathbb T$ in $\R^2$ and the infinite supercritical percolation cluster $\cC_\infty$ in $\Z^2$. We proved in this article that both $\mathbb T, \cC_\infty$ satisfies items $(i)$ and $ (iii)$ from \cref{sec:UST} and we proved $(ii)$ in \cref{thm:RSW_perc,thm:RSW_Poisson}. Finally, Quenched invariance principle is also known for both these cases, see \cref{rmk:quenched}. Thus we obtain
\begin{corollary}\label{cor:GFF}
The conclusions of \cref{thm:GFF} hold if $\mu$ is either the unique infinite cluster of supercritical percolation, or the Poisson Voronoi triangulation.
\end{corollary}

\bibliographystyle{abbrv}
\bibliography{RSW_quant}

\begin{thebibliography}{10}

\bibitem{AP_chemical}
P.~Antal and A.~Pisztora.
\newblock {On the chemical distance for supercritical Bernoulli percolation}.
\newblock {\em The Annals of Probability}, 24(2):1036 -- 1048, 1996.

\bibitem{Barlow_perc}
M.~T. Barlow.
\newblock {Random walks on supercritical percolation clusters}.
\newblock {\em The Annals of Probability}, 32(4):3024 -- 3084, 2004.

\bibitem{BLR_riemann2}
N.~Berestycki, B.~Laslier, and G.~Ray.
\newblock The dimer model on {R}iemann surfaces {II}: convergence of height
  function.
\newblock 2019.
\newblock Preprint.

\bibitem{BLR16}
N.~Berestycki, B.~Laslier, and G.~Ray.
\newblock {Dimers and imaginary geometry}.
\newblock {\em The Annals of Probability}, 48(1):1 -- 52, 2020.

\bibitem{BLR_riemann1}
N.~Berestycki, B.~Laslier, and G.~Ray.
\newblock The dimer model on {R}iemann surfaces, {I}: convergence of
  {T}emperleyan forests.
\newblock Preprint, 2021.

\bibitem{berger2007quenched}
N.~Berger and M.~Biskup.
\newblock Quenched invariance principle for simple random walk on percolation
  clusters.
\newblock {\em Probability theory and related fields}, 137(1-2):83--120, 2007.

\bibitem{duminil2017lectures}
H.~Duminil-Copin.
\newblock Lectures on the ising and potts models on the hypercubic lattice.
\newblock In {\em PIMS-CRM Summer School in Probability}, pages 35--161.
  Springer, 2017.

\bibitem{duminil_sixty}
H.~Duminil-Copin.
\newblock Sixty years of percolation.
\newblock In {\em Proceeding of the International Congress}. World Scientific,
  2018.

\bibitem{Garet_chemical}
O.~Garet and R.~Marchand.
\newblock {Large deviations for the chemical distance in supercritical
  Bernoulli percolation}.
\newblock {\em The Annals of Probability}, 35(3):833 -- 866, 2007.

\bibitem{grimmett1999percolation}
G.~Grimmett.
\newblock What is percolation?
\newblock In {\em Percolation}, pages 1--31. Springer, 1999.

\bibitem{KPWTemperley}
R.~W. Kenyon, J.~G. Propp, and D.~B. Wilson.
\newblock Trees and matchings.
\newblock {\em Electron. J. Combin.}, 7, 2000.

\bibitem{liggett1997domination}
T.~M. Liggett, R.~H. Schonmann, and A.~M. Stacey.
\newblock Domination by product measures.
\newblock {\em The Annals of Probability}, 25(1):71--95, 1997.

\bibitem{LyonsPeres}
R.~Lyons and Y.~Peres.
\newblock {\em Probability on trees and networks}, volume~42.
\newblock Cambridge University Press, 2017.

\bibitem{rousselle15}
A.~Rousselle.
\newblock {Quenched invariance principle for random walks on Delaunay
  triangulations}.
\newblock {\em Electronic Journal of Probability}, 20(none):1 -- 32, 2015.

\bibitem{wilson96}
D.~B. Wilson.
\newblock Generating random spanning trees more quickly than the cover time.
\newblock In {\em Proceedings of the twenty-eighth annual ACM symposium on
  Theory of computing}, pages 296--303, 1996.

\bibitem{YY}
A.~Yadin and A.~Yehudayoff.
\newblock Loop-erased random walk and {P}oisson kernel on planar graphs.
\newblock {\em Ann. Probab.}, 39(4):1243--1285, 2011.

\end{thebibliography}
\end{document}